\title{A behavioral interpretation of belief functions}
\date{\today}
\author{Timber Kerkvliet and Ronald Meester \\ \\ VU University Amsterdam}

\documentclass[11pt,twoside,a4paper]{article}

\usepackage{amsthm}
\usepackage{amsmath}
\usepackage{amssymb}

\numberwithin{equation}{section}

\newtheorem{theorem}{Theorem}[section]

\theoremstyle{definition}
\newtheorem{definition}[theorem]{Definition}

\theoremstyle{remark}
\newtheorem{example}[theorem]{Example}

\newcommand{\B}{{\rm Bel}}

\newcommand{\Bu}{{\rm Buy}}
\newcommand{\Se}{{\rm Sell}}

\begin{document}
\maketitle
\begin{abstract} 
Shafer's belief functions were introduced in the seventies of the previous century as a mathematical tool in order to model epistemic probability. One of the reasons that they were not picked up by mainstream probability was the lack of a behavioral interpretation. In this paper we provide such a behavioral interpretation, and re-derive Shafer's belief functions via a betting interpretation reminiscent of the classical Dutch Book Theorem for probability distributions. We relate our betting interpretation of belief functions to the existing literature.
\end{abstract}

\section{Introduction}
\label{sec:intro}

Modern mainstream probability theory, be it aleatory or epistemic, is almost exclusively based on the axioms of Kolmogorov, and we start our exposition with quoting these axioms. We fix a finite set  $\Omega$ of outcomes.

\begin{definition}[Kolmogorov axioms]
A function $P: 2^{\Omega} \rightarrow [0,1]$ is a probability distribution if it satisfies:
\begin{enumerate}
\item [(P1)] $P(\Omega)=1$ and $P(\emptyset)=0$;
\item [(P2)] For every every $A, B \subset \Omega$ we have
\begin{equation}
P(A \cup B) = P(A) + P(B) - P(A \cap B).
\end{equation}
\end{enumerate}
\end{definition}

From an aleatory point of view, the axioms are often justified via a frequentistic interpretation of probabilities. In such a frequentistic interpretation, we take relative frequencies in repeated experiments as the motivation and justification of the axioms. 

From an epistemic point of view, the probability of an event $A$ is often interpreted as the degree of belief an agent has in a $A$. This degree of belief can be quantified as the price for which the agent is willing to both buy and sell a bet that pays out $1$ if $A$ turns out to be true. Using the Dutch Book argument this interpretation also leads to the Kolmogorov axiomatization, as is well known. 

However, degrees of belief cannot always be satisfactorily described with the classical axioms of Kolmogorov, something which has been recognized and confirmed by many researchers from such different disciplines as mathematics \cite{shafer76, shafer81, shafer08, cohen77, walley}, legal science \cite{cohen77}, and philosophy (see \cite{haenni, dietz10, FH, walley} and references therein). These authors have argued that the classical axioms of probability are too restrictive for at least two reasons:

\begin{enumerate}
\item It is impossible for an agent to distinguish between disbelief in $A$, by which we mean $P(A^c)$, and lack of belief, by which we mean $1-P(A)$. Especially when we interpret degree of belief in $A$ as the degree to which an agent has supporting evidence for $A$, lack of supporting evidence for $A$ is not necessarily supporting evidence for $A^c$.
\item It is impossible for an agent to suspend all judgment, that is, assigning $P(A)=P(A^c)=0$ is impossible. But certainly it is possible to have no supporting evidence at all for either $A$ or $A^c$. As an example, consider a situation in which a man claims to be the father of a certain child. A DNA test may or may not rule out the man as a potential father of the child, but in any classical probabilistic computation one has to start with a prior probability for the man to be the father. Often one takes a fifty-fifty prior, but clearly this does not correspond to our knowledge. Since we have no evidence for either fathership or non-fathership, it would be more reasonable to assign zero prior belief to both possibilities, something which is impossible under the Kolmogorov axioms. (Uniform prior probabilities are also problematic from another point of view: lumping states together or changing the scale does not preserve uniformity, as is well known, see e.g.\ \cite{shafer76}.)
\end{enumerate}

These and similar considerations motivated Glenn Shafer \cite{shafer76} to introduce belief functions, which were supposed to better capture the nature of epistemic probability. To see how belief functions differ from classical probability distributions, we note that it is well known that (P2) can be expanded, for any collection of events $A_1, \ldots, A_N$, into the well-known inclusion-exclusion formula
\begin{equation}
\label{eq:beliefadditivity}
P\left(\bigcup_{i=1}^N A_i \right) = \sum_{\substack{I \subseteq \{1,\ldots,N\} \\ I \not= \emptyset}} (-1)^{|I|+1}P \left( \bigcap_{i\in I} A_i \right).
\end{equation}

In the Shafer axiomatization of belief functions, (\ref{eq:beliefadditivity}) is replaced by a corresponding \emph{inequality}.
\begin{definition}[Shafer axioms (\cite{shafer76})]
\label{def:belieffunction}
A function $\B: \Omega \rightarrow [0,1]$ is a \emph{belief function} provided that
\begin{itemize}
\item[(B1)] $\B(\emptyset)=0$ and $\B(\Omega)=1$;
\item[(B2)] For all $A_1,A_2,\ldots, A_N \subset \Omega$, we have
\begin{equation}
\label{eq:condB2*}
\B\left(\bigcup_{i=1}^N A_i \right) \geq \sum_{\substack{I \subseteq \{1,\ldots,N\} \\ I \not= \emptyset}} (-1)^{|I|+1}\B \left( \bigcap_{i\in I} A_i \right).
\end{equation}
\end{itemize}
\end{definition}

One may at this point already loosely argue that the Shafer axioms are indeed suitable for describing degrees of belief, as follows. Taking $N=2$ in (\ref{eq:condB2*}) for simplicity gives $\B(A \cup B) \geq \B(A) + \B(B) - \B(A \cap B)$. A loose interpretation now is as follows. If we have supporting evidence for either $A$ or $B$, then of course we have also supporting evidence of $A \cup B$. However, supporting evidence for both $A$ and $B$ is then counted twice and must be subtracted on the right hand side. Now notice that it is possible that there is supporting evidence for $A \cup B$, but not for $A$ or $B$ individually, leading to the inequality in the formula. 

Shafer showed that belief functions have the following characterization.

\begin{definition}
A function $m: 2^{\Omega} \rightarrow [0,1]$ is called a {\em basic belief assignment} if $m(\emptyset)=0$ and
\begin{equation}
\label{eq:additivitym}
\sum_{C \subseteq \Omega} m(C)=1.
\end{equation}
\end{definition}

\begin{theorem}
\label{def:belief}
A function $\B: 2^\Omega \rightarrow [0,1]$ is a belief function if and only if there exists a basic belief assignment $m$ such that
\begin{equation}
\B(A) = \sum_{C \subseteq A} m(C).
\end{equation}
There is a one-to-one correspondence between belief functions and basic belief assignments, and $\B$ is a probability distribution if and only if $m$ concentrates on singletons.
\end{theorem}
In many situations, there is a natural candidate for the basic belief assignment $m$, especially when the latter is obtained through a classical probabilistic experiment. 

Although belief functions are used in certain applied settings (see \cite{cuzzolin} for a recent text), researchers in mainstream mathematics have essentially stayed away from it. There are, roughly speaking, two reasons for this. First, there were many critics (see the references in \cite{shafer81}) for which a theory about uncertainty without a behavioral (betting) interpretation was unacceptable. 

A second major point of concern is formulated in \cite{schum94} and \cite{fs86}. In both references, the main reason to reject Shafer's belief functions is that the calculus of these belief functions, as put forward in the so called Dempster rule of combination to combine two belief functions into a new one, would not be well founded or motivated, and would lead to unacceptable results. 

The main goal of the current paper is to show how belief functions arise form a natural betting interpretation, thereby taking away the first major point of concern. We already mentioned above that classically, the degree of belief an agent has in an event $A$, is interpreted as the price for which he or she is willing to both buy and sell a bet that pays out $1$ if $A$ turns out to be true. We argue, however, that a degree of belief an agent has in an event $A$, should be interpreted as the maximum price for which he or she is willing to \emph{buy} (not necessarily sell) a bet that pays out $1$ if $A$ turns out to be true. In our approach, the difference between buying and selling is interpreted as the difference between disbelief and lack of believe that we mentioned above. The distinction between buying and selling prices also appears in the theory of Peter Walley \cite{walley}, and below we comment on the relation between his general approach and our theory.

We do not further comment on Dempster's rule of combination for the simple reason that we have no need for this rule, and we develop the theory without it.

Other characterizations of belief functions exist in the literature, for instance in the work of Smets \cite{smets}. We think our approach is more direct than his, and as a result our characterization is much more concise (he has 10 requirements). Moreover, unlike Smets, we work in a betting interpretation, so that our result can be seen as an analogue of the Dutch Book Theorem.

The paper is built up as follows. In Section \ref{sec:bet} we develop our betting theory, and in Section \ref{sec:main} we state and prove the betting interpretation of belief functions. This section contains our main result Theorem \ref{belangrijk}. 

\section{Betting functions and degrees of belief}
\label{sec:bet}
We approach belief behaviorally, working under the assumption that the degree to which someone believes an event, can be measured by his or her willingness to accept bets. To go towards a definition that we can work with mathematically, we introduce betting functions. 

\begin{definition}
A \emph{bet} (or gamble) on $\Omega$ is a function $X: \Omega \rightarrow \mathbb{R}$. We write $\mathcal{X} = \mathbb{R}^X$ for the collection of all bets on $\Omega$. A \emph{betting function} is a function $R: \mathcal{X} \rightarrow \{0,1\}$ such that for each $X \in \mathcal{X}$, there is an $\alpha_0 \in \mathbb{R}$ such that $R(X+\alpha)=0$ for $\alpha<\alpha_0$ and $R(X+\alpha)=1$ for $\alpha \geq \alpha_0$. 
\end{definition}

We interpret $R$ as a function that indicates, for each bet $X \in \mathcal{X}$, whether or not an agent is willing to accept $X$, where we interpret $R(X)=1$ as `willing to accept the bet' and $R(X)=0$ `not willing to accept the bet'. The definition of a betting function justifies the following definition.

\begin{definition}
Let $R$ be a betting function. The \emph{buy function} $\mathrm{Buy}_R: \mathcal{X} \rightarrow \mathbb{R}$ is given by
\begin{equation}
\mathrm{Buy}_R(X) := \max \{ \alpha \in \mathbb{R} \;:\; R(X - \alpha)=1 \}. 
\end{equation}
This is the maximum price an agent is willing to pay for a bet which pays out $X$. The \emph{sell function}  $\mathrm{Sell}_R: \mathcal{X} \rightarrow \mathbb{R}$ is given by
\begin{equation}
\mathrm{Sell}_R(X) := \min \{ \alpha \in \mathbb{R} \;:\; R(\alpha - X)=1 \}. 
\end{equation}
This is the minimum price an agent is willing to sell the bet which pays out $X$.
\end{definition}

The buy and sell function have the following relation:
\begin{equation}
\begin{aligned}
\Bu_R(X) & = \max\{ \alpha \in \mathbb{R} \;:\; R(X-\alpha)=1 \} \\
& = \max\{  -\alpha \in \mathbb{R} \;:\; R(X+\alpha)=1 \} \\
& = - \min\{ \alpha \in \mathbb{R} \;:\; R(\alpha-(-X))=1 \} \\
& = - \Se_R(-X).
\end{aligned}
\end{equation}
This shows how buy and sell functions are dual in the sense that $\Bu_R$ is completely determined by $\Se_R$ and vice versa. This relation between $\Bu$ and $\Se$ is precisely the relation between lower and upper previsions in \cite{walley}, but in our case the relation follows from the underlying concept of a betting function. Note that we can recover $R$ from $\Bu_R$, since
\begin{equation}
R(X) = \left\{
	\begin{array}{ll}
		1  & \mbox{if } \Bu_R(X) \geq 0, \\
		0 & \mbox{if } \Bu_R(X) < 0.
	\end{array}
\right.
\end{equation}

\begin{definition}
\label{coherent}
Let  $R: \mathcal{X} \rightarrow \{0,1\}$ be a betting function. Then $R$ is said to be {\em coherent} if
\begin{itemize} 
\item[(R1)] If $X>0$, then $R(X)=1$;
\item[(R2)] $R(\lambda X) = R(X)$ for every $\lambda>0$;
\item[(R3)] If $R(X)=R(Y)=1$, then $R(X+Y)=1$.
\end{itemize}
\end{definition}

These conditions do not necessarily capture everything a reasonable agent should adhere to, and for rational behavior more is needed as we will see. However, the conditions formulate the very basics. Condition (R1) says that agents should be willing to accept bets with guaranteed positive results. Condition (R2) says that willingness to accept bets should only depend on the relative sizes of the  results, but not on their absolute sizes. Condition (R3) says that if an agent is willing to accept two bet, (s)he should be willing to accept bets simultaneously. We want to point out that both condition (R2) and (R3) have a debatable consequence: if an agent is willing to accept a bet in which (s)he wins 1 euro if $A$ is true and loses 1 euro if $A$ is false, then (s)he) should be willing to accept a bet in which (s)he wins 1 million euro if $A$ is true and loses 1 million euro if $A$ is false. In the real world, one could think of all kinds of reasons why an agent would not want to accept the second bet, even if (s)he is willing to accept the first one. By imposing coherence, we thus consider highly idealized agents.

The following result says that $\Bu_R$ is coherent as a lower prevision in the sense of Walley \cite{walley} if and only if $R$ is coherent in the sense of Definition \ref{coherent}.

\begin{theorem}
\label{thm:coherence1}
Let $R: \mathcal{X} \rightarrow \{0,1\}$ be a betting function. Then $R$ is coherent if and only if
\begin{itemize} 
\item $\Bu_R(X) \geq \min X$;
\item $\Bu_R(\lambda X)= \lambda \Bu_R(X)$ for every $\lambda>0$;
\item $\Bu_R(X+Y) \geq \Bu_R(X) + \Bu_R(Y)$.
\end{itemize}
\end{theorem}

\begin{proof}
Suppose $R$ is coherent. By (R1), we have that $R(X - \min X + \epsilon)=1$ for every $\epsilon>0$. Hence
\begin{equation}
\Bu_R(X) = \max \{ \alpha \in \mathbb{R} \;:\; R(X-\alpha)=1 \} \geq \min X - \epsilon
\end{equation}
for every $\epsilon>0$, so $\Bu_R(X) \geq \min X$. By (R2), we have for every $\lambda>0$ that
\begin{equation}
\begin{aligned}
\Bu_R(\lambda X) & = \max \{ \alpha \in \mathbb{R} \;:\; R(\lambda X-\alpha)=1 \} \\
& = \max \{ \alpha \in \mathbb{R} \;:\; R(X- \frac{\alpha}{\lambda})=1 \} \\
& = \max \{ \lambda\alpha \in \mathbb{R} \;:\; R(X- \alpha)=1 \} \\
& = \lambda \max \{ \alpha \in \mathbb{R} \;:\; R(X- \alpha)=1 \} \\
& = \lambda \Bu_R(X).
\end{aligned}
\end{equation}
Suppose that, for some $\alpha,\beta \in \mathbb{R}$, we have $R(X-\alpha)=1$ and $R(Y-\beta)=1$. Then, by (R3), we have $R(X+Y-\alpha-\beta)=1$ and thus
\begin{equation}
\Bu_R(X+Y) = \max \{ \gamma \in \mathbb{R} \;:\; R(X+Y-\gamma)=1 \} \geq \alpha + \beta.
\end{equation}
It follows that
\begin{equation}
\begin{aligned}
\Bu_R(X+Y) & \geq  \max\{ \alpha \;:\; R(X-\alpha)=1 \} +  \max\{ \beta \;:\; R(X-\beta)=1 \} \\
& = \Bu_R(X) + \Bu_R(Y).
\end{aligned}
\end{equation}

For the converse, suppose that $\Bu_R$ has the listed properties. We set $\psi: \mathbb{R} \rightarrow \{0,1\}$ by
\begin{equation}
\psi(x) = \left\{
	\begin{array}{ll}
		1  & \mbox{if } x \geq 0, \\
		0 & \mbox{if } x < 0,
	\end{array}
\right.
\end{equation}
and we note as before that
\begin{equation}
R(X) = \psi(\Bu_R(X)).
\end{equation}

If $X>0$, then $\Bu_R(X) \geq 0$ by the first property. Hence $R(X)= \psi(\Bu_R(X))=1$. So (R1) holds.

The second property tells us that, for every $\lambda>0$, we have
\begin{equation}
R(\lambda X) = \psi(\Bu_R( \lambda X)) = \psi( \lambda \Bu_R(X) ) = \psi( \Bu_R(X) ) = R(X),
\end{equation}
so (R2) holds.

By the third property, we have
\begin{equation}
\begin{aligned}
R(X+Y) & = \psi(\Bu_R(X+Y)) \\
& \geq \psi(\Bu_R(X) + \Bu_R(Y)) \\
& \geq \min \{ \psi(\Bu_R(X)), \psi(\Bu_R(Y)) \} \\
& = \min \{ R(X), R(Y) \}.
\end{aligned}
\end{equation}
Hence if $R(X)=R(Y)=1$, it follows that $R(X+Y)=1$, so (R3) holds.
\end{proof}

The second and third property of Theorem \ref{thm:coherence1} tell us that $\Bu_R$ is a super-linear functional if $R$ is coherent. Coherence of $R$ can of course also be captured in terms of $\Se_R$:

\begin{theorem}
\label{thm:coherence2}
Let $R: \mathcal{X} \rightarrow \{0,1\}$ be a betting function. Then $R$ is coherent if and only if
\begin{itemize} 
\item $\Se_R(X) \leq \max X$;
\item $\Se_R(\lambda X)= \lambda \Se_R(X)$ for every $\lambda>0$;
\item $\Se_R(X+Y) \leq \Se_R(X) + \Se_R(Y)$.
\end{itemize}
\end{theorem}

\begin{proof}
This follows directly from Theorem \ref{thm:coherence1} and the relation between $\Bu_R$ and $\Se_R$.
\end{proof}

We want to measure the degree to which an agent believes $A \subseteq \Omega$ by the willingness to accept a bet with a desirable result if $A$ is true and an undesirable result if $A$ is false. The actions that have a desirable result if $A$ is true and an undesirable result if $A$ is false, are buying the bet $1_A$ and selling the bet $1_{A^c}$. It is willingness to buy $1_A$ for a high price and willingness to sell $1_{A^c}$ for a low price, that shows a high degree of belief in $A$. This leads to our definition of degree of belief.

\begin{definition} [Degree of belief]
Let $R$ be the coherent betting function of an agent. We define the degree to which this agent believes an event $A \subseteq \Omega$ as $\Bu_R(1_A)=1-\Se_R(1_{A^c})$.
\end{definition}

\section{Adding B-consistency}
\label{sec:main}

In this section, we introduce an additional condition for betting functions and show that this constraint precisely leads to buy functions which are belief functions when restricted to bets of the form $1_A$, with $A \subset \Omega$. Before we do this, however, we will briefly discuss how our setup relates to the Dutch Book argument for probability distributions.

The Dutch Book argument is centered around the principle that  betting behavior of agents should not lead to sure loss. The following theorem tells us that not having a sure loss is already implied by coherence. This theorem is well known within Walley's theory, but we give our version of the proof as a service to the reader.

\begin{theorem}
\label{thm:coherence2}
If $R$ is a coherent betting function, then 
\begin{equation}
\label{eq:nosureloss}
\max_{\omega \in \Omega} \left(\sum_{i=1}^N \left( X_i - \Bu_R(X_i)\right) + \sum_{j=1}^M \left( \Se_R(Y_j) - Y_j\right) \right) \geq 0
\end{equation}
for all $X_1,\ldots,X_N \in \mathcal{X}$ and $Y_1,\ldots,Y_M \in \mathcal{X}$.
\end{theorem}

\begin{proof}
We first show that $\Bu_R(X) \leq \max X$ for each $X \in \mathcal{X}$. Suppose that $\Bu_R(X) > \max X$. This means there is an $\epsilon>0$ such that
\begin{equation}
R(X- \max X - \epsilon) = 1.
\end{equation}
Note that for every $Y \in \mathcal{X}$, there is a $\lambda>0$ such that $\lambda(X-\max  X -\epsilon) < Y$. Since $R(\lambda(X- \max X - \epsilon)) = 1$ by (R2) and $R(Y-\lambda(X-\max  X -\epsilon))=1$ by (R1), it follows with (R3) that $R(Y)=1$. Hence $R(Y)=1$ for all $Y \in \mathcal{X}$, but then there is no maximum $\alpha \in \mathbb{R}$ such that $R(Y-\alpha)=1$. This is a contradiction, and it follows that $\Bu_R(X) \leq \max X$.

Now let $X_1,\ldots,X_N \in \mathcal{X}$ and $Y_1,\ldots,Y_M \in \mathcal{X}$. By using Theorem \ref{thm:coherence1} and the property we just proved, we find
\begin{equation}
\begin{aligned}
\sum_{i=1}^N \Bu_R(X_i) - \sum_{j=1}^M \Se_R(Y_j) & = \sum_{i=1}^N \Bu_R(X_i) + \sum_{j=1}^M \Bu_R(-Y_j) \\
& \leq \Bu_R \left( \sum_{i=1}^N X_i + \sum_{j=1}^M -Y_j \right) \\
& \leq \max \left( \sum_{i=1}^N X_i + \sum_{j=1}^M -Y_j \right).
\end{aligned}
\end{equation}
\end{proof}

Theorem \ref{thm:coherence2} tells us that coherence is stronger than having no sure loss. At the same time, even coherence does not imply that $A \mapsto \Bu_R(1_A)$ is a probability distribution: it is clear that the collection of maps $A \mapsto \Bu_R(1_A)$ for coherent $R$, is much richer than only probability distributions. This tells us that the property that $\Bu_R(1_A) = \Se_R(1_A)$ for every $A \subseteq \Omega$, which is usually implicitly assumed when the Dutch Book argument is laid out, is crucial for the restriction to probability distributions. The next theorem confirms this.

\begin{theorem}
$P$ is a probability distribution if and only if there exists a coherent betting function $R$ such that $P(A)=\Bu_R(1_A)=\Se_R(1_A)$.
\end{theorem}

\begin{proof}
Suppose there is a coherent $R$ such that $P(A)=\Bu_R(1_A) = \Se_R(1_A)$. We have $P(\Omega)=\Bu_R(1_\Omega)=1$ and $P(A)=\Bu_R(1_A) \geq 0$ by coherence. Then for disjoint $A,B \subseteq \Omega$ we find
\begin{equation}
P(A \cup B) = \Bu_R(1_A + 1_B) \geq \Bu_R(1_A) + \Bu_R(1_B) = P(A)+P(B)
\end{equation}
and
\begin{equation}
P(A \cup B) = \Se_R(1_A + 1_B) \leq \Se_R(1_A) + \Bu_R(1_B) = P(A)+ P(B).
\end{equation}
Hence $P$ is additive.

For the converse, suppose that $P$ is a probability distribution. Then define
\begin{equation}
\Bu_R(X) = \sum_{\omega \in \Omega} P(\{ \omega \}) X(\omega).
\end{equation}
Clearly, $R$ is coherent and we have $\Bu_R(X)=\Se_R(X)$.
\end{proof}

As we already mentioned, the constraint that $\Bu_R(1_A)=\Se_R(1_A)$ for every $A \subseteq \Omega$, is precisely one we do not want to impose. This property, however, is not easily weakened in a reasonable way. Therefore, we will work towards another characterization of probability distributions (Theorem \ref{hierboven}), from which we can derive our constraint. We start with the definition of a {\em belief valuation}.

\begin{definition}
\label{def:belval}
A function ${\cal B}: 2^{\Omega} \to \{0,1\}$ is called a \emph{belief valuation} provided that
\begin{itemize}
\item ${\cal B}(A^c)=0$ if ${\cal B}(A)=1$;
\item If ${\cal B}(A)=1$ and $A \subseteq B$, then ${\cal B}(B)=1$;
\item If ${\cal B}(A)= 1 $ and ${\cal B}(B)=1$, then ${\cal B}(A \cap B)=1$;
\item ${\cal B}(\Omega)=1$.
\end{itemize}
\end{definition}

A belief valuation is also called a categorical belief function or a 0-1 necessity measure in the literature. In practice, we use the characterization that $\mathcal{B}$ is a belief valuation if and only if there is a nonempty set $E_\mathcal{B} \subseteq \Omega$ such that
\begin{equation}
    \mathcal{B}(A) = \left\{ \begin{array}{ll}
		1  & \mbox{if } E_\mathcal{B} \subseteq A \\
		0 & \mbox{otherwise }
		\end{array}
		\right. .
\end{equation}
We can also describe belief valuations in terms of filters, since $\mathcal{B}$ is a belief valuation if and only if
\begin{equation}
\{ A \subseteq \Omega \;:\; \mathcal{B}(A)=1 \} \subseteq 2^\Omega
\end{equation}
is a proper filter of subsets of $\Omega$. This filter can be interpreted as the collection of sets in which an agent has full belief. The next result makes this precise.  For $R$ a coherent betting function, we denote by $\mathcal{B}_R: 2^{\Omega} \to \{0,1\}$ the function that satisfies $\mathcal{B}_R(A)=1$ if and only if $\Bu_R(1_A)=1$.

\begin{theorem}
A function $\mathcal{B}: 2^{\Omega} \to \{0,1\}$ is a belief valuation if and only if there is a coherent betting function $R$ such that $\mathcal{B}=\mathcal{B}_R$.
\end{theorem}

\begin{proof}
Suppose first that $\cal{B}$ is a belief valuation. We define $R$ by
$$
\Bu_R(X):= \min_{\omega \in E_\mathcal{B}} X(\omega).
$$
Clearly $R$ is coherent and it follows from the definition of $\mathcal{B}_R$ that $\mathcal{B}_R= \mathcal{B}$.

For the converse, suppose that $R$ is a coherent betting function. We check that $\mathcal{B}_R$ is a belief valuation. Since $\Bu_R(1_{\Omega})=1$ we have $\mathcal{B}_R(\Omega)=1$. The second property in Definition \ref{def:belval} follows immediately from the monotonicity of $\Bu_R$. If $\Bu_R(A)=\Bu_R(B)=1$, then   
\begin{eqnarray*}
\Bu_R(1_{A \cap B}) &=& \Bu_R(1_A + 1_B - 1_{A \cup B})\\
&\geq & \Bu_R(1_A) + \Bu_R(1_B)+ \Bu_R(-1_{A \cup B}) \geq 1,
\end{eqnarray*}
so $\Bu_R(1_{A \cap B})=1$. Finally, if $\mathcal{B}_R(A)=1$, then since
$$
1= \Bu_R(1_A + 1_{A^c}) \geq \Bu_R(1_A) + \Bu_R(1_{A^c}),
$$ 
it follows that $\Bu_R(1_{A^c})=0$.
\end{proof}

 A belief valuation should be compared with the classical notion of a {\em truth valuation}. The definition of a truth valuation $\mathcal{T}:2^{\Omega} \to \{0,1\}$ is similar to the definition of a belief valuation, the only difference being that in the first bullet, `if' is replaced by `if and only if'. Hence a truth valuation is a special belief valuation, namely one that corresponds to a proper \emph{ultra}filter of sets. Truth valuations are precisely those belief valuations $\mathcal{B}$ for which $E_\mathcal{B}$ is a singleton. A major difference between truth and belief is that if an agent does not believe in $A$, this does not imply that (s)he does believe in $A^c$. As a result, the implication in the first bullet in the definition of a belief valuation goes in one direction only. 

Given a belief valuation $\mathcal{B}$ and a set $S \subseteq \Omega$ for which $\mathcal{B}(S)=1$, the agent fully believes that a bet $X \in \mathcal{X}$ has a revenue of at least
\begin{equation}
\min_{\omega \in S} X(\omega).
\end{equation}
Since this holds for all $S$ for which $\mathcal{B}(S)=1$, this leads to the definition of {\em guaranteed revenue}.

\begin{definition}
For any belief valuation $\cal{B}$, the {\em guaranteed revenue} $G_{\cal{B}}: \cal{X} \to \mathbb{R}$ is defined as
$$
G_{\cal{B}}(X)= \max_{A: \mathcal{B}(A) = 1} \min_{\omega \in A} X(\omega).
$$
\end{definition}
Since we have that $\mathcal{B}(A)=1$ if and only if $E_\mathcal{B} \subseteq A$, we can express the guaranteed revenue as
\begin{equation}
G_{\cal{B}}(X) = \min_{ \omega \in E_\mathcal{B}} X(\omega). 
\end{equation}

To benchmark and motivate our main result Theorem \ref{belangrijk} below we now first show how classical probability distributions can be characterized with the notion of guaranteed revenue.

\begin{definition}
A betting function $R$ is \emph{P-consistent} if and only if, for all $X_1,\ldots,X_N$ and $Y_1,\ldots,Y_M$ such that
\begin{equation}
\label{eq:Pcons1}
G_\mathcal{B} \left( \sum_{i=1}^N  X_i \right) \leq G_\mathcal{B} \left( \sum_{j=1}^M Y_j \right)
\end{equation}
for every belief valuation $\mathcal{B}$, we have
\begin{equation}
\label{eq:Pcons2}
\sum_{i=1}^N \Bu_R(X_i) \leq \sum_{j=1}^M \Bu_R(Y_j).
\end{equation}
\end{definition}

\begin{theorem}
\label{hierboven}
$P$ is a probability distribution if and only if there exists a coherent and P-consistent $R$ such that $P(A)=\Bu_R(1_A)$.
\end{theorem}

In words, this result says that if the guaranteed revenue of one collection of bets is larger than the guaranteed revenue of a second collection of bets, then the agent should be willing to pay more for the second collection.

The proof of the theorem below reveals that the statement of the theorem is, strictly speaking, overly complicated. Indeed, if we would leave out $G_{\mathcal{B}}$ everywhere, the ensuing result would still be true, and probably easier to interpret: the theorem would say that if one collection of bets always pays out more than a second collection, an agent would be willing to pay more for the first collection. We have chosen for the current formulation since this formulation points the way for the necessary changes.

\begin{proof} {\em (of Theorem \ref{hierboven}.)}
First suppose that $R$ is coherent, P-consistent and that $P(A)=\Bu_R(1_A)$. Suppose $A \cap B = \emptyset$. Then 
\begin{equation}
G_\mathcal{B}(1_A + 1_B) = G_\mathcal{B}(1_{A \cup B})
\end{equation}
for every $\mathcal{B}$, so by P-consistency we have
\begin{equation}
\Bu_R(1_A + 1_B) = \Bu_R(1_A) + \Bu_R(1_B).
\end{equation}
Hence $P(A \cup B) = P(A) + P(B)$. Since $R$ is coherent, we have $P(\Omega)=\Bu_R(1_\Omega)=1$ and it follows that $P$ is a probability distribution.

Now suppose that $P$ is a probability distribution. We define $R$ by
\begin{equation}
\Bu_R(X) := \sum_{\omega \in \Omega} X(\omega) P(\{\omega\}).
\end{equation}
Clearly $\Bu_R(1_A)=P(A)$, and it follows from Theorem \ref{thm:coherence1} that $R$ is coherent. Now suppose that
\begin{equation}
\label{eq:everyB}
G_\mathcal{B} \left( \sum_{i=1}^N  X_i \right) \leq G_\mathcal{B} \left( \sum_{j=1}^M Y_j \right)
\end{equation}
for every belief valuation $\mathcal{B}$. For every $\omega \in \Omega$, the map $\mathcal{B}_\omega(A) := 1_A(\omega)$ is a belief valuation. Note that $G_{B_\omega}(X) = X(\omega)$, so (\ref{eq:everyB})  implies that
\begin{equation}
\sum_{i=1}^N  X_i \leq  \sum_{j=1}^M Y_j.
\end{equation}
Hence, by our definition of $\Bu_R$:
\begin{equation}
\begin{aligned}
\sum_{i=1}^N  \Bu_R(X_i) & = \Bu_R \left( \sum_{i=1}^N X_i \right) \\
& \leq \Bu_R \left( \sum_{j=1}^M Y_j \right) \\
& = \sum_{j=1}^M \Bu_R(Y_j).
\end{aligned}
\end{equation}
So $R$ is P-consistent.
\end{proof}

Although we do not deny that the notion of P-consistency is in some sense a reasonable requirement for collections of bets, there is, from the point of view of epistemic probability, a problem with it. Whereas in (\ref{eq:Pcons1}) the guaranteed revenues of the \emph{combined} bets are compared, in (\ref{eq:Pcons2}) the sums of the buy prices of the \emph{individual} bets are compared. This observation goes back to the heart of the problem with the use of classical probability distributions for epistemic purposes: the guaranteed revenue of the combined bets $1_A$ and $1_A^c$ is - of course - the same as the guaranteed revenue of the bet $1$, under any belief valuation. But this should not have any direct implication for the maximum price for which an agent would be willing to buy $1_A$ or $1_A^c$ individually. 

This suggests that for epistemic purposes, we should change P-consistency in one of the following two ways: in (\ref{eq:Pcons1}) the sums of the guaranteed revenues of the individual bets should be compared, or in (\ref{eq:Pcons2}) the buy prices of the combined bets should be compared. The first option leads to our definition of B-consistency. 

\begin{definition}
A betting function $R$ is \emph{B-consistent} if for all $X_1,\ldots,X_N$ and $Y_1,\ldots,Y_M$ such that
\begin{equation}
\sum_{i=1}^N G_\mathcal{B}(X_i) \leq \sum_{j=1}^M G_\mathcal{B}(Y_j)
\end{equation}
for every belief valuation $\mathcal{B}$, we have
\begin{equation}
\sum_{i=1}^N \Bu_R(X_i) \leq \sum_{j=1}^M \Bu_R(Y_j).
\end{equation}
\end{definition}

There is a simple reason for this choice. Indeed, the alternative would result in the constraint that if $G_\mathcal{B}(\sum_i X_i) \leq G_\mathcal{B}(\sum_j Y_j)$ for every belief valuation $\mathcal{B}$, then $\Bu_R(\sum X_i) \leq \Bu_R(\sum_j Y_j)$. This constraint, however, is simply B-consistency for $N=M=1$. 

Hence, the notion of B-consistency differs from P-consistency in the sense that we compare the correct quantities: we compare sums of guaranteed revenues of individual (collective) bets to sums of buy prices of individual (collective) bets. 

The next example illustrates that not every coherent betting function is B-consistent.

\begin{example}
\label{ex:noBcons}
Suppose $\Omega = \{1,2,3,4\}$ and that $R$ is given by
\begin{equation}
\Bu_R(X) = \min \left\{ \frac{1}{2} \sum_{i=1}^2 X(i), \frac{1}{4} \sum_{i=1}^4 X(i) \right\}.
\end{equation}
It is easy to check that $R$ is coherent, and that for all belief valuations $\cal{B}$ we have
\begin{equation}
    G_\mathcal{B}(1_{\{2,3,4\}}) + G_\mathcal{B}(1_{\{2\}}) \geq G_\mathcal{B}(1_{\{2,3\}}) + G_\mathcal{B}(1_{\{2,4\}}),
\end{equation}
but
\begin{equation}
\Bu_R(1_{\{2,3,4\}}) + \Bu_R(1_{\{2\}}) = \frac{3}{4} < 1 = \Bu_R(1_{\{2,3\}}) + \Bu_R(1_{\{2,4\}}).
\end{equation}
\end{example}

With the notion of B-consistency we can now state and prove our main result. The following theorem constitutes our behavioral interpretation of belief functions. It shows that only B-consistency is needed to guarantee that a lower prevision in the sense of Walley \cite{walley} is in fact a belief function and that every belief function can be obtained this way. The result legitimizes the use of belief functions for modelling epistemic probability. 

\begin{theorem}
\label{belangrijk}
$\B$ is a belief function if and only if there exists a coherent and B-consistent $R$ such that $\B(A) = \Bu_R(1_A)$.
\end{theorem}

This result follows immediately from the following theorem which is interesting in its own right and characterizes B-consistency for coherent betting functions.

\begin{theorem}
\label{thm:karak}
Let $R$ be a coherent betting function. Then $R$ is B-consistent if and only if there is a basic belief assignment $m$ such that
\begin{equation}
\label{eq:choquet}
\Bu_R(X)=\sum_{S \subseteq \Omega} m(S) \min_{\omega \in S} X(\omega)
\end{equation}
for all $X \in \cal{X}$.
\end{theorem}

The expression in (\ref{eq:choquet}) is known as the Choquet integral of $X$ with respect to the belief function that corresponds to $m$ (see for instance \cite{grabisch}), and is also called the \emph{lower expectation} of $X$ (see for instance \cite{shafer08}). Hence, another way of phrasing Theorem \ref{thm:karak} is that $R$ is B-consistent if and only if $\B(A) := \Bu_R(1_A)$ is a belief function and 
$\Bu_R(X)$ is given by the Choquet integral of $X$ with respect to $\B$. We now give the proof.

\begin{proof}[Proof of Theorem \ref{thm:karak}]
First suppose that there is a basic belief assignment $m$ such that
$$
\Bu_R(X)=\sum_{S \subseteq \Omega} m(S) \min_{\omega \in S} X(\omega)
$$
for all $X \in \cal{X}$. Also suppose that for $X_1,\ldots,X_N,Y_1,\ldots,Y_M \in \mathcal{X}$ we have
\begin{equation}
    \sum_{i=1}^N \min_{\omega \in S} X_i(\omega) \leq \sum_{j=1}^M \min_{\omega \in S} Y_j(\omega)
\end{equation}
for every nonempty $S \subseteq \Omega$. Then
\begin{equation}
\begin{aligned}
\sum_{i=1}^N \Bu_R(X_i) & = \sum_{i=1}^N  \sum_{S \subseteq \Omega} m(S) \min_{\omega \in S} X_i(\omega) \\
& =  \sum_{S \subseteq \Omega} m(S)  \sum_{i=1}^N \min_{\omega \in S} X_i(\omega) \\
& \leq  \sum_{S \subseteq \Omega} m(S)  \sum_{j=1}^M \min_{\omega \in S} Y_j(\omega) \\
& =  \sum_{j=1}^M \Bu_R(Y_j).
\end{aligned}
\end{equation}
Hence $R$ is B-consistent.

For the converse of the theorem, suppose that $R$ is B-consistent. First we show that $\B(A):=\Bu_R(1_A)$ is a belief function. Clearly we have
\begin{equation}
G_\mathcal{B}(1_A) = \mathcal{B}(A),
\end{equation}
and since $\mathcal{B}$ is a belief function, we have
\begin{equation}
\mathcal{B}\left(\bigcup_{i=1}^N A_i \right) + \sum_{\substack{I \subseteq \{1,\ldots,n\} \\ I \not= \emptyset, |I| \mathrm{\;even}}} \mathcal{B} \left( \bigcap_{i \in I} A_i \right) \geq \sum_{\substack{I \subseteq \{1,\ldots,n\} \\ I \not= \emptyset, |I| \mathrm{\;odd}}} \mathcal{B} \left( \bigcap_{i \in I} A_i \right).
\end{equation}

So by B-consistency, we have
\begin{equation}
\B\left(\bigcup_{i=1}^N A_i \right) + \sum_{\substack{I \subseteq \{1,\ldots,n\} \\ I \not= \emptyset, |I| \mathrm{\;even}}} \B \left( \bigcap_{i \in I} A_i \right) \geq \sum_{\substack{I \subseteq \{1,\ldots,n\} \\ I \not= \emptyset, |I| \mathrm{\;odd}}} \B \left( \bigcap_{i \in I} A_i \right), 
\end{equation}
and it follows that $\B$ is a belief function.

Since $\B$ is a belief function, there is a $m: 2^\Omega \rightarrow [0,1]$ with $m(\emptyset)=0$ such that
\begin{equation}
\B(A) = \sum_{S \subseteq A} m(S).
\end{equation}
Now let $X \in \mathcal{X}$. We write $X(\Omega) = \{y_1,\ldots,y_K\}$ where $y_1 < y_k < \ldots < y_K$. Set $y_0:=0$. It is well known that the Choquet integral of $X$ can be expressed as
\begin{equation}
\label{eq:een}
\sum_{S \subseteq \Omega} m(S) \min_{\omega \in S} X(\omega) =  \sum_{j=1}^K  (y_j - y_{j-1}) \B( \{X \geq y_j\}).
\end{equation}
We have
\begin{equation}
\begin{aligned}
\sum_{k=1}^K \min_{\omega \in S} (y_k - y_{k-1}) 1_{\{X \geq y_k\}}(\omega)  & = \sum_{k=1}^K (y_k - y_{k-1}) 1( S \subseteq \{X \geq y_k \}) \\
& = \min_{\omega \in S} X(\omega).
\end{aligned}
\end{equation}
So by B-consistency we find that
\begin{equation}
\label{eq:twee}
\sum_{k=1}^K  \Bu_R((y_k - y_{k-1}) 1_{\{X \geq y_k\}})) = \Bu_R(X).
\end{equation}
Now it follows with (\ref{eq:een}), (\ref{eq:twee}) and coherence of $R$ that
\begin{equation}
\Bu_R(X) = \sum_{S \subseteq \Omega} m(S) \min_{\omega \in S} X(\omega).
\end{equation}
\end{proof}

\section{Discussion}
We have first argued that the classical axioms of Kolmogorov are not suitable for modeling epistemic probability. This has been known for a long time, and researchers like Glenn Shafer and Peter Walley have developed a general theory of belief functions, respectively lower provisions, as an alternative for classical probability theory in an epistemological context. 

The theory of belief functions was never picked up in mainstream probability. One of the reasons for this was the lack of a clear behavioral interpretation of belief functions, analogous to the betting interpretation of probabilities. In this paper we have developed such a behavioral interpretation. We have embedded belief functions in a betting context, and we have shown that belief functions arise precisely when we add B-consistency to the coherent lower previsions in the theory of Walley. In this way, not only do we provide a behavioral interpretation of belief functions, but we also make a natural connection between the theory of Walley and the theory of Shafer.

Of course adding B-consistency calls for an argument why a rational agent should adhere to it. We think it is not controversial to say that no guaranteed losses (Theorem \ref{thm:coherence2}) is the bottom line for any reasonable constraint. Beyond that things are, of course, debatable. Our argument to restrict the lower previsions of Walley to B-consistency is derived from the way we obtained it, namely by altering P-consistency in a very reasonable way. If a collection of bets is in some sense guaranteed to be better than another collection of bets, then the total price should be higher. The point is now how ``better'' should be formulated. When compared to P-consistency, B-consistency allows for all the flexibility of epistemic probability that we asked for in the introduction. Indeed, note that it is possible to set $m(\Omega)=1$, which corresponds to total ignorance apart from the fact that the outcome is in $\Omega$. We think it is rational to compare collection of bets from the viewpoint of total guaranteed revenue, and to be willing to pay more when this quantity increases. As such, we think that B-consistency is a very reasonable constraint for a rational agent to adhere to. 

But there is more than philosophy, of course. We also want a theory that is practical and relatively easy to apply. Belief functions are close to classical probabilities since they are determined by basic belief assignments. As such, we think they are much more practical and easier to use than coherent lower previsions. Indeed, in practical situations, one does not directly use constraints as coherence or B-consistency to construct a buy function. In case of belief functions, one typically proceeds by constructing a suitable basic belief assignment, see for instance our paper \cite{kerkmees} in which we apply belief functions in the classical forensic context of the so called island problem, precisely by setting up appropriate basic belief assignments. This is very much akin the classical situation: not many people use characterizations like P-consistency (or related characterizations) to set up a probability measure, but it is reassuring that such characterizations exist. Hence we should view B-consistency as a theoretical underpinning for why to use belief functions, but not as a tool that is used in practice to construct belief functions. 

Last but not least, we mention conditional belief functions, a notion which we have not introduced in the current paper. Shafer originally based his notion of conditional belief on the so called Dempster rule of combination, a rule that has been criticized fiercely, see e.g.\ \cite{pearl, FH}. In a forthcoming paper we will discuss conditional beliefs from the perspective of the current betting interpretation. It turns out that there are various ways to do this, depending on the precise epistemological situation, and this fact adds to the suitableness of belief functions to model epistemic uncertainty. Another natural line of research is to further develop the theory of belief functions in infinite spaces, an enterprise that already has obtained some attention in e.g.\ \cite{smets2}.

\medskip\noindent
{\bf Acknowledgement:} We thank an anonymous referee for valuable comments, most notably for his/her suggestion to relate our results to the theory of Peter Walley.

\end{document}